\documentclass[12pt]{amsart}
\usepackage{amsmath}
\usepackage{amssymb}
\usepackage{amsthm}
\usepackage[english]{babel}
\usepackage[all]{xy}
\usepackage{geometry}
\usepackage{txfonts}
\usepackage{geometry}
\newtheorem{teo}{Theorem}[section]
\newtheorem{coro}[teo]{Corollary}
\newtheorem{lem}[teo]{Lemma}
\newtheorem{pro}[teo]{Proposition}

\newtheorem{defn}[teo]{Definition}
\newtheorem*{teonn}{Theorem}

\newcommand{\D}{\mathbb{D}}
\newcommand{\B}{\mathbb{B}}
\newcommand{\HH}{\mathbb{H}}
\newcommand{\N}{\mathbb{N}}
\newcommand{\C}{\mathbb{C}}
\newcommand{\rr}{\mathbb{R}}
\newcommand{\s}{\mathbb{S}}
\newcommand{\Z}{\mathbb{Z}}

\DeclareMathOperator{\prodstella}{\medcirc \hskip -1.45ex \star }

\DeclareMathOperator{\prodstellanorma}{ \|_{H^2(\B)\medcirc \hskip -1.05ex \star H^2(\B)} }
\DeclareMathOperator{\prodstellanormaduale}{ \|_{(H^2(\B)\medcirc \hskip -1.05ex \star H^2(\B))^*} }

\newcommand{\de}{\partial_c}
\newcommand{\p}{\partial}

\DeclareMathOperator{\IIm}{Im}

\DeclareMathOperator{\RRe}{Re}

\DeclareMathOperator{\ext}{ext}
\DeclareMathOperator{\supp}{supp}

\newcommand{\RR}{\mathbb{R}}
\newcommand{\BB}{\mathbb{B}}

\newcommand{\II}{\mathbb{I}}
\renewcommand{\SS}{\mathbb{S}}
\newcommand{\BBB}{\mathcal{B}}

\title{ From Hankel operators to Carleson measures in a quaternionic variable}

\author{Nicola Arcozzi} 
\address{Dipartimento di Matematica, Universit\`a di Bologna, Piazza di Porta San Donato 5, 40126 Bologna, Italy} 
\email{nicola.arcozzi@unibo.it} 
\author{Giulia Sarfatti} 
\address{Dipartimento di Matematica, Universit\`a di Bologna, Piazza di Porta San Donato 5, 40126 Bologna, Italy} 
\email{giulia.sarfatti@unibo.it, giulia.sarfatti@unifi.it }

\begin{document}
\begin{abstract}
We introduce and study Hankel operators defined on the Hardy space of regular functions of a quaternionic variable.
Theorems analogous to those of Nehari anc C. Fefferman are proved.
\end{abstract}

\maketitle


{
\small
\noindent \bf Notation. \it The symbol $\HH$ denotes the set of the quaternions $q=x_0+x_1i+x_2j+x_3k=\RRe(q)+\IIm(q)$, with $\RRe(q)=x_0$ and $\IIm(q)=x_1i+x_2j+x_3k$;  
where the $x_j$'s are real numbers and the imaginary units
$i,j,k$ are subject to the rules $ij=k,\ jk=i,\ ki=j$ and $i^2=j^2=k^2=-1$. We identify the quaternions $q$ whose imaginary part vanishes, $\IIm(q)=0$, 
with real numbers, $\RRe(q)\in\RR$; 
and, similarly, we let $\II=\RR i+\RR j+\RR k$ be the set of the imaginary quaternions. The norm $|q|\ge0$ of $q$ is 
$|q|=\sqrt{\sum_{l=0}^3 x_l^2}=(q\overline{q})^{1/2}$, where $\overline{q}=x_0-x_1i-x_2j-x_3k$ is the conjugate of $q$. 
The open unit ball $\BB$ in $\HH$ contains the quaternions $q$ such that $|q|<1$. The boundary of $\BB$ in $\HH$
is denoted by $\partial\BB$. By the symbol $\SS$ we denote the unit sphere of the imaginary quaternions: $q\in\II$ belongs to $\SS$ if $|q|=1$. For $I$ in $\s$, the slice 
$L_I=L_{-I}$ in $\HH$ contains all quaternions having the form $q=x+yI$, with $x,y$ in $\RR$.
\rm
}

\section{Introduction}
Let $\HH$ be the skew-field of the quaternions.
The quaternionic Hardy space $H^2(\BB)$ consists of the formal power series of the quaternionic variable $q$, $|q|<1$,
\begin{equation}
\label{serie}
 f(q)=\sum_{n=0}^\infty q^na_n,
\end{equation}
such that the sequence of quaternions $\{a_n\}$ satisfies
\begin{equation}\label{htwo}
\|f\|_{H^2(\BB)}:=\|\{a_n\}\|_{\ell^2(\N,\HH)}= \left(\sum_{n=0}^\infty \left|a_n\right|^2\right)^{1/2}<\infty.
\end{equation}
Such functions are \it regular \rm in $\B$ in the sense of Gentili and Struppa \cite{GSAdvances}.
For functions in $H^2(\BB)$, the quaternion valued inner product is
\begin{equation*} 
\left\langle\sum q^na_n,\sum q^nb_n\right\rangle_{H^2(\BB)}:=\sum_{n=0}^\infty\overline{b_n}a_n.
\end{equation*}
We might think of $a=\{a_n\}$ as a discrete, positive time, quaternionic-valued signal, 
with finite energy; and, using standard notation, of $f(q)=\widehat{a}(q)=\sum_{n=0}^{\infty}q^na_n$ as its generating function. Suppose $T$ is a linear, stable, 
time-invariant, realizable filter acting on such signals. By this we mean a linear, bounded operator $T$ on $\ell^2(\N,\HH)$, such that: \it (time-invariance) \rm 
$T$ commutes with the shift operator;
and \it(realizability) \rm the output $Ta$ at time $n$ depends on $a_m$ with $m\le n$ only.   The shift operator $S$ is defined as
$$
Sa(n)=a(n-1).
$$
We have used the functional notation $a(n)=a_n$.
We adopt here the convention that in vector spaces $V$ on the quaternionic skew-field, in multiplication, scalar factors are on the right: if $v\in V$ and $\alpha\in\HH$, then $v\alpha\in V$.
Exactly as in the complex valued case, $T$ is time-invariant and well defined on Kronecker functions $\delta_k$ ($k\in\Z$) if, and only if, $T=T_K$ is a convolution operator:
\begin{equation}\label{ascoli}
T_Ka(n)=K*a(n)=\sum_{k\in\Z}K(n-k)a(k),
\end{equation}
with $K=\{K(n)\}$ the sequence defined by $K=T\delta_0$. On the other hand, realizability is equivalent to the condition that $K(n)=0$ if $n<0$.
Relation \eqref{ascoli} justifies the introduction of the $\star$-product between functions of the form \eqref{serie}. If $g(q)=\sum_{n\ge0}q^nb_n$, then
$$
(f\star g)(q)=\sum_{n\ge0}q^n\sum_{k=0}^na_{n-k}b_{k}.
$$
The $\star$-product is not a pointwise product. See Section \ref{prelim}. 
If $\phi(q)=\widehat{K}(q)= \sum_{n\ge0}q^nK(n)$, then \eqref{ascoli} becomes
\begin{equation}\label{bolzano}
\widehat{Ta}=\phi\star\widehat{a}:=M_\phi a: 
\end{equation}
$M_\phi$ is a $\star$-multiplication operator,
where 
$\widehat{a}(q)=\sum_{n\ge0}q^na(n)$.

Let $\BBB(\mathcal{H})$ be the noncommutative algebra of the bounded operators on the quaternionic Hilbert space $\mathcal{H}$, 
normed with the sup-norm. See \cite{ghilonimorettiperotti} for the basic facts of functional analysis of linear spaces over the quaternions. 
The sharp result on boundedness for $\star$-multiplication 
operators is in \cite{milanesi3}.
\begin{teonn}[A]  Let $\phi:\BB\to\HH$ be a function which can written  as $\phi(q)=\sum_{n\ge0}q^n\check{\phi}(n)$, where the sum converges absolutely in $\BB$. Then, 
 $$
 \|M_\phi\|_{\BBB(H^2(\BB))}=\|\phi\|_{H^\infty(\BB)}:=\sup_{q\in\BB}|\phi(q)|.
 $$
 As a consequence, if $\check {\phi}$ denotes the sequence of coefficients of $\phi$, $\check{\phi}=\{ \check{\phi} (n)\}_{n}$, then $ \|T_{\check{\phi}}\|_{\mathcal B (\ell^2(\N,\HH))}=\|\phi\|_{H^\infty(\BB)}$.
\end{teonn}
Equation \eqref{ascoli}, with $n,k\ge0$, might be interpreted as multiplication times an infinite matrix $A_{\alpha}=[\alpha(n,k)]_{n,k\ge0}=[K(n-k)]_{n,k\ge0}$ with constant diagonals.
In applications, it is important to consider the case of an infinite matrix with constant anti-diagonals,
that is of Hankel operators. See \cite{peller} for a detailed excursion into the theory of such operators.

Let $\alpha:\N\to\HH$ be a quaternion valued sequence. Form the infinite matrix $A_\alpha=[\alpha(j+k)]_{j,k=0}^{+\infty}$, 
and let $\Gamma_\alpha$ act on $\HH$ valued sequences $v=(v_j)_{j=0}^{+\infty}$ by matrix multiplication: $(\Gamma_\alpha v)(j)=\sum_{k=0}^{\infty}\alpha(j+k)v(k)$. 

\vskip0.5cm

\noindent \bf Nehari's Problem. \rm Under which conditions on $\alpha$ is the operator $\Gamma_\alpha$ bounded on $\ell^2(\N,\HH)$? 

\vskip0.5cm

\noindent Exactly as in the complex valued case, the problem can be reformulated in terms of regular functions; one has only to be careful in keeping account of the non-commutativity of
the product in $\HH$. 
For two positive real numbers $r,s$ we will write $r\approx s$ if $\frac{1}{C} r \le s \le C r$ for a positive constant $C$ and we will write $r\lesssim s$ if there exists $C>0$ such that $r\le Cs$. 
Our first main result is the following.

\begin{teo}\label{eboli}
Let the sequence $\alpha$ and the regular function $b(q)=\sum_{n=0}^\infty q^n\check{b}(n)$ be related via
\begin{equation}\label{domodossola}
\alpha(n)=\overline{\check{b}(n)},\text{ for }n\ge0.
\end{equation}
Set $\Lambda_b(h):=\left\langle h,b\right\rangle_{H^2(\B)}$ when $h$ is regular in $\B$. 

The following conditions are equivalent.
\begin{enumerate}
 \item[(a)] $ \left\|\Gamma_\alpha\right\|_{\BBB(\ell^2(\N,\HH))}<\infty$.
 \item[(b)]  $\sup_{f,g\ne0}\frac{\left|\left<f\star g,b\right>_{H^2(\BB)}\right|}{\|f\|_{H^2(\BB)}\cdot\|g\|_{H^2(\BB)}}<\infty$.
 \item[(c)] $\|b\|_{BMO(\BB)}:=\sup_{I\in\SS}\|b\|_{BMO(\BB_I)}<\infty$; where $\|b\|_{BMO(\BB_I)}$ is the $BMO$-norm of the restriction of $b$ to the slice $\BB_I$.  
 \item[(d)] The measure $d\mu_b(q)=|\de b(q)|^2(1-|q|^2)dVol_\BB(q)$ on $\BB$ is a Carleson measure for $H^2(\BB)$:  the inequality $\int_\BB|f|^2d\mu_b\le c\left(\mu_b\right)\|f\|_{H^2(\BB)}^2$
 holds whenever $f$ is in $H^2(\BB)$.
 \item[(e)] $\Lambda_b$ belongs to $\left[H^2(\BB)\prodstella H^2(\BB)\right]^*$, the space dual to the weak $\star$-product of two copies of $H^2(\BB)$, w.r.t. 
 the inner product in $H^2(\BB)$.
\end{enumerate}
 Moreover,
 $$
 \left\|\Gamma_\alpha\right\|_{\BBB(\ell^2(\N,\HH))}= \sup_{f,g\ne0}\frac{\left|\left<f\star g,b\right>_{H^2(\BB)}\right|}{\|f\|_{H^2(\BB)}\cdot\|g\|_{H^2(\BB)}}\approx 
 \|\Lambda_b\prodstellanormaduale\approx \|b\|_{BMO(\BB)}\approx C\left(\mu_b\right)^{1/2}.
 $$
\end{teo}
In (d), $\de b(q)=\sum_{n=1}^\infty n\check{b}(n)q^{n-1}$ is the Cullen derivative of $b$. See Section \ref{prelim} for basic terminology and facts regarding regular functions.

The statement summarizes a circle of ideas which, in the complex valued case, are due to Nehari \cite{nehari}, C. Fefferman \cite{feffermanstein}, 
Coifman, Rochberg and Weiss \cite{coifmanrochbergweiss}. The weak $\star$-product  $H^2(\BB)\prodstella H^2(\BB)$ is defined as the space of the functions  $\varphi$
which are regular in $\BB$, and such that
\begin{equation}\label{firenze}
 \|\varphi\prodstellanorma:=\inf\left\{\sum_{j=1}^\infty\|f_j\|_{H^2(\BB)}\cdot\|g_j\|_{H^2(\BB)}:\ \varphi=\sum_{j=1}^\infty f_j\star g_j\right\}<\infty.
\end{equation}
The proof of Theorem \ref{eboli} partly makes direct use of the corresponding results in one complex variable, and partly adapts their proofs to the quaternionic setting, 
but some new ideas are required.
On the one hand there is no available inner-outer factorization for $H^2$ functions in the quaternionic sense. On the other hand the geometric properties of $H^2$'s reproducing kernel are not as clear as in the one complex variable case.

In Section \ref{gorizia} we discuss several equivalent definitions of $BMOA$ in the quaternionic context. There we show in a rather indirect way the curious fact that in (\ref{firenze})
the infimum might be taken over decompositions with two summands only, $\varphi=\sum_{j=1}^2 f_j\star g_j$. Having just one summand would imply a good factorization for 
regular functions in $H^1(\B)$, but we do not have a result of this strength.

In order to make more concrete the definition of Carleson measures, we also provide a geometric characterization of them, which might have some independent interest.
For $q=re^{J\theta}\in \B$, with $r\ge0$, $\theta\in\rr$, and $J\in\SS$, let 
\[S(q)= \{\varrho e^{I\alpha}\in \B \ \colon \ |\alpha-\theta|\le1-r,\  \,  0<1-\varrho \le 2(1- r),\ I\in\SS\}.\]
be the symmetric box in $\B$, indexed by $q$, which is independent of the particular $J$.
\begin{teo}\label{badolo}
A measure $\mu$ on $\B$ is a Carleson measure for the Hardy space $H^2(\B)$ if and only if for any $q\in \B$, the measure of the symmetric box $S(q)$ satisfies 
\[\mu(S(q))\lesssim 1-|q|.\] 
\end{teo}
The Theorem corresponds to Carleson's characterization of the corresponding measures on the complex unit disc, see \cite{carleson}. The ``only if'' part of the proof requires some care.
The fact that it suffices to test the measure over symmetric boxes only reflects the fact that regular functions are affine with respect to the imaginary unit, or, more geometrically, the
fact that the ``invariant metric'' for functions regular on $\B$ is bounded on copies of $\s$ inside $\B$ (see \cite{metrica}).

The paper is structured as follows. In Section \ref{prelim} we provide some background material on regular functions of a quaternionic variable. Sections \ref{carlesonsection}, with
 the proof of Theorem \ref{badolo}, and Section \ref{hankelsection}, with the characterization of the symbols of bounded Hankel operators, are essentially independent of each other.  
 In Section \ref{gorizia} we discuss some properties 
of the quaternionic version of analytic $BMO$. Some questions we could not answer are summarized in Section \ref{trento}.

\section{Preliminaries}\label{prelim}
In this section we recall the definition of slice regular functions over the quaternions $\HH$ (in the sequel simply ``regular'' functions), together with some basic properties. We will restrict our attention to functions defined on the quaternionic unit ball $\B=\{q\in \HH \ : \ |q|<1\}$.
We refer to the book \cite{libroGSS} for all details and proofs.

Let $\s$ denote the two-dimensional sphere of imaginary units of $\HH$, $\s=\{q\in \HH \, | \, q^2=-1\}$. 
One can ``slice'' the space $\HH$ in copies of the complex plane that intersect along the real axis,  
\[ \HH=\bigcup_{I\in \s}(\rr+\rr I),  \hskip 1 cm \rr=\bigcap_{I\in \s}(\rr+\rr I),\]
where $L_I:=\rr+\rr I\cong \C$, for any $I\in\s$.  
Each element $q\in \HH$ 
can be expressed as $q=x+yI_q$, where $x,y$ are real (if $q\in\rr$, then $y=0$) and $I_q$ is an imaginary unit.  
To have a unique decomposition (outside the real axis) we choose $y\geq 0$.

A function $f:\B \to \HH$ is called {\em (slice) regular} if for any $I\in\s$ the restriction $f_I$ of $f$ to $\B_I=\B\cap L_I$
is {\em holomorphic}, i.e. it has continuous partial derivatives and it is such that
\[\overline{\p}_If_I(x+yI)=\frac{1}{2}\left(\frac{\p}{\p x}+I\frac{\p}{\p y}\right)f_I(x+yI)=0\]
for all $x+yI\in \B_I$.

A wide class of examples of regular functions is given by power series with quaternionic coefficients of the form $\sum^{\infty}_{n= 0}q^na_n$
which converge in open balls centered at the origin. In fact,
a function $f$ is regular on $\B$
if and only if $f$ has a power series expansion
$f(q)=\sum^{\infty}_{n= 0}q^na_n$ converging in $\B$.

The {\em slice (or Cullen) derivative} of a function $f$ which is regular on $\B$, is the regular function defined as 
\[\de f(x+yI)=\frac{1}{2}\left(\frac{\p}{\p x} -I\frac{\p}{\p y}\right)f_I(x+yI).\]

%
Slice regular functions defined on $\B$ present a peculiar property.
\begin{teo}[Representation and Extension Formula]\label{RF}
Let $f:\B\to \HH$ be a regular function and let $x+y\s\subset \B$. Then, for any $I,J\in\s$,
\[f(x+yJ)=\frac{1}{2}[f(x+yI)+f(x-yI)]+J \frac I 2 [f(x-yI)-f(x+yI)].\]
Moreover, the previous formula allows one to uniquely extend any holomorphic function $g_I:\B_I\to \HH$ to a slice regular function on $\B$, denoted by $\ext(g_I)$. 
\end{teo}

A basic result that establishes a relation between regular functions and holomorphic functions of one complex variable is the following.
\begin{lem}[Splitting Lemma]\label{split}
Let $f$ be a regular function on $\B$. Then for any $I\in\s$ and for any $J\in \s$, $J\perp I$ there exist two holomorphic functions $F,G:\B_I=\B\cap L_I\to L_I$ such that
\[f(x+yI)=F(x+yI)+G(x+yI)J\]
for any $x+yI\in\B_I$.
\end{lem} 
Since the pointwise product of functions does not preserve slice regularity, 
a new multiplication operation for regular functions is defined. In the special case of power series,
the {\em regular product} (or {\em $\star$-product}) 
of $f(q)=\sum_{n=0}^{\infty}q^na_n$ and $g(q)=\sum_{n=0}^{\infty}q^nb_n$ is
\[
f\star g(q)=\sum_{n\ge0}q^n \sum_{k=0}^{n}a_kb_{n-k}.\]

\noindent The $\star$-product is related to the standard pointwise product by the following formula.
\begin{pro}\label{trasf}
Let $f,g$ be regular functions on $\B$. Then
\[f\star g(q)=\left\{\begin{array}{l r}
0 & \text{if $f(q)=0 $}\\
f(q)g(f(q)^{-1}qf(q)) & \text{if $f(q)\neq 0 $}
                \end{array}\right.
 \] 
\end{pro}

\noindent The  reproducing kernel for the quaternionic Hardy space $H^2(\B)$  has the expression:
\[k_{w}(q)=\sum_{n=0}^{\infty}q^n\overline{w}^n=(1-q\overline{w})^{-\star}.\]
where the power  $-\star$ denotes the reciprocal in the slice regular sense, \cite{libroGSS}.
See \cite{milanesi3} for this and other basic facts concerning $H^2(\B)$.
\section{Carleson measures on $\B$}\label{carlesonsection}
A nonnegative Borel measure $\mu$ on $\B$ is called a {\em Carleson measure} for the Hardy space $H^2(\B)$ if it satifies the imbedding inequality
\[\int_{\B}|f(q)|^2d\mu(q)\le c(\mu)\| f\| ^2_{H^2(\B)}\]
for any $f\in H^2(\B)$, with a constant $c(\mu)$ depending on $\mu$ alone.

We can always decompose a measure $\mu$ on the unit ball as $\mu=\mu_1+\mu_2$ where $\mu_1(\B \cap \rr)=0$ and $\supp(\mu_2)\subseteq \B \cap \rr$. Moreover $\mu$ is a 
Carleson measure if and only if the measures $\mu_1$ and $\mu_2$ are Carleson as well.

Thanks to the Disintegration Theorem (see Theorem 2.28 in \cite{ambrosio}), any finite measure $\mu$ on $\B$, such that $\mu(\B\cap\rr)=0$,  
can be uniquely decomposed as
\begin{equation}\label{radon}
d\mu(x+yI)=
d\mu^+_I(x+yI)d\nu(I)
\end{equation}
where $\nu$ is the measure on the sphere $\s$ defined by 
\[\nu(E)=\mu\left(\{x+yI \in \B\ |\ y>0 \text{ and } I\in E\}\right)\] 
\noindent and $\mu^+_I$ is a (probability) measure on $\B^+_I=\{x+yI\in \B_I \, : \, y \ge 0\}$. 
Hence we can write 
\[\int_{\B}\varphi(x+yI)d\mu(x+yI) = \int_{\s}\int_{\B^+_{I}}\varphi(x+yI) d\mu^+_{I}(x+yI) d\nu(I)\]
for any $\varphi: \B \to \HH$. If in general $\mu$ is a finite measure on $\B$, 
we can decompose $\mu=\mu_\rr+\tilde \mu$, where $\tilde \mu(\B\cap \rr)=0$ and $\supp \mu_\rr \subseteq \B\cap \rr$, 
so that
\begin{equation*}\label{krypton}\int_{\B}\varphi(x+yI)d\mu(x+yI) = \int_{\B\cap \rr}\varphi(x)d\mu_{\rr}(x)+\int_{\s}\int_{\B^+_{I}}\varphi(x+yI) d\tilde\mu^+_I(x+yI)d\nu(I)\tag{3.2},\end{equation*}
where $\tilde\mu^+_I$ is obtained from $\tilde\mu$ using the Disintegration Theorem.

A stronger property for a measure defined on the unit ball $\B$ is to be a Carleson measure on each slice $\B_I$.
For any $I\in \s$, let $\mu_I=\mu_\rr+\tilde\mu^+_{I}+\tilde\mu^+_{-I}$ be the ``restriction'' of $\mu$ to the slice $\B_I$. 

 A finite Borel measure $\mu$ is called {\em slice Carleson} for $H^2(\B)$ if there exists a constant $c(\mu)$ such that
\begin{align*}\label{slicecarleson}
 &\int_{\B_{I}}|f(x+yI)|^2d\mu_I(x+yI) \tag{3.3}\\
 &=\int_{\B\cap \rr}|f(x)|^2d\mu_{\rr}(x)+\int_{\B^+_{I}}|f(x+yI)|^2 d\tilde\mu^+_I(x+yI)+\int_{\B^+_{-I}}|f(x+y(-I))|^2 d\tilde\mu^+_{-I}(x+y(-I))\\
 &\le c(\mu)\| f\| ^2_{H^2(\B)} ,
\end{align*}
%
for any $f\in H^2(\B)$ and $I$ in $\SS$. 

For any $I\in \s$, the norm $\| f\| ^2_{H^2(\B)}$ can be interpreted as the complex $H^2$-norm of the restriction $f_I(x+yI)=\sum_{n=0}^{\infty}(x+yI)^na_n$ of $f$ to $\B_I$:
\[\| f\| ^2_{H^2(\B_I)}=\sum_{n=0}^\infty|a_n|^2=\| f\| ^2_{H^2(\B)}.\]


\begin{pro}\label{sliceimplica}
Let $\mu$ be a finite Borel measure. If $\mu$ is slice Carleson for the Hardy space $H^2(\B)$, then it is also Carleson. 
\end{pro}
\begin{proof}
Using the  same notation as in \eqref{krypton}, since $\mu(\BB)=\nu(\s)<+\infty$ we can apply the Disintegration Theorem to write 
\begin{eqnarray*}\int_{\B}|f(q)|^2d\mu(q)&=&\int_{\B\cap \rr}|f(x)|^2d\mu_{\rr}(x)+ \int_{\s}d\nu(I)\int_{\B^+_{I}}|f(z)|^2d\tilde \mu^+_I(z)\crcr&
\lesssim& \| f\| ^2_{H^2(\B)}+ \nu(\s)\| f\| ^2_{H^2(\B)}\lesssim \| f\| ^2_{H^2(\B)}\end{eqnarray*}
for any $f\in H^2(\B)$.

\end{proof}
It is not difficult to find examples of infinite measures $\mu$ on $\B$ such that $\mu_I$ is a Carleson measure on each slice
$\B_I$.

As in the classical complex setting (see, e.g., \cite{Duren}), it is possible to give a characterization of Carleson measures in terms of ``boxes''.
For $q=re^{J\theta}\in \B$, denote by $A_I(q)$ the arc of $\p \B_I$ defined as
\[A_I(q)=\left\{e^{I\alpha}\in \p \B_I \ \colon  \ |\alpha - \theta| \le  1-r \right\},\]
and let $S_I(q)$ be the ``box'' in $\B_I$ defined by 
\[S_I(q)= \{\varrho e^{I\alpha}\in \B_I \ \colon \ e^{I\alpha}\in A_I(q), \,  0<1-\varrho \le 2(1- r)\}.\]
 \begin{defn}
Let $q\in \B$. The symmetric subset of $\B$ obtained as 
\[S(q)=\cup_{I\in\s}S_I(q)\] 
is called a {\em symmetric box}. 
\end{defn}
For slice Carleson measures the characterization follows quite easily from the classical case.
\begin{pro}\label{charslice}
A finite Borel measure $\mu$ on $\B$ is slice Carleson for $H^2(\B)$ if and only if for any $I\in \s$ and  $z\in\B_I$, $\mu_I(S_I(z))\lesssim |A_I(z)|$, 
where $|A_I(z)|$ denotes the length of the arc $A_I(z)$.
\end{pro}
The main ingredients of the proof are the classical Carleson Theorem (see, e.g., Theorem 9.3 in \cite{Duren}) and the Splitting Lemma \ref{split}.
%

The characterization theorem for general Carleson measures requires the following technical result.
\begin{lem}\label{stimamoebius}
Let $w\in \D$ and let 
$s(w)=\{z \in \D \colon \ 1-|z| \le 2(1-|w|), |\arg(z)-\arg(w)|\le 1-|w|\}$. 
Then there exists a constant $c>0$, such that 
\[\frac{1}{c}\le \left|\frac{1-z\frac{w+\bar w}{2}}{1-zw}\right|\le c\]
for any $z\in s(w)$. 
\end{lem}
\begin{proof}
We can suppose $\RRe w >0$, since the statement is invariant by multiplication by $-1$, and $\IIm w \ge 0$, since $\overline{s(w)}=s(\bar w)$. 
We will write $a+ib\approx \alpha+i\beta$ if $\frac{1}{c}|\alpha|\le \left| a \right|\le c |\alpha|$ and $\frac{1}{c}|\beta|\le \left| b \right|\le c|\beta|$ for some positive constant $c$. Hence $a+ib\approx \alpha+i\beta$ implies $|a+ib|\approx |\alpha+i\beta|$.\\

\noindent (a) Consider first the case where $w$ is near the real axis.
Let $w=(1-\varepsilon )e^{i\delta}$ and $z=(1-\varepsilon ')e^{i(\delta'+\delta)}$ where $0<\varepsilon \le \varepsilon _0$ for a fixed $\varepsilon _0>0$,  $0\le \delta \le \delta_0$ and $|\delta'|, \varepsilon '\le  \varepsilon $.
For $\varepsilon _0$ and $\delta_0$ small enough we have $\varepsilon +\varepsilon '+\varepsilon  \varepsilon ' \approx \varepsilon +\varepsilon '$ and $(2\delta+\delta')^2+|2\delta+\delta'|\approx |2\delta+\delta'|$, so that
\begin{align*}
1-zw&=1-(1-\varepsilon )e^{i\delta}(1-\varepsilon ')e^{i(\delta+\delta')}=1-(1-\varepsilon )(1-\varepsilon ')+(1-\varepsilon )(1-\varepsilon ')\left(1-e^{i(2\delta + \delta')}\right)\\
&\approx \varepsilon  + \varepsilon ' + 1-\cos(2\delta+\delta')-i\sin(2\delta+\delta')\approx  \varepsilon  + \varepsilon ' +(2\delta+\delta')^2-i(2\delta+\delta').
\end{align*}
With regard to the norm:
\begin{align*}
|1-zw|\approx \varepsilon  + \varepsilon ' +(2\delta+\delta')^2+|2\delta+\delta'|\approx \varepsilon+ |2\delta+\delta'|.
\end{align*}
We have also
\begin{align*}
& 1-z\frac{w+\bar w}{2}=1-(1-\varepsilon )(1-\varepsilon ')e^{i(\delta+\delta')}\cos \delta\\
&=1-(1-\varepsilon )(1-\varepsilon ')+(1-\varepsilon )(1-\varepsilon ')\left(1-\cos \delta\right) +(1-\varepsilon )(1-\varepsilon')\cos \delta \left(1-e^{i(\delta+\delta')}\right)\\
&\approx \varepsilon+\varepsilon'+\delta^2+1-\cos(\delta+\delta')-i\sin(\delta+\delta')
\end{align*}
which, in terms of norm, implies
\begin{align*}
\left|1-z\frac{w+\bar w}{2}\right|&\approx \varepsilon+\varepsilon'+\delta^2+(\delta+\delta')^2+|\delta+\delta'|\approx \varepsilon+\delta^2+|\delta+\delta'|. 
\end{align*}
Moreover, if 
$|2\delta+\delta'|\ge 5\varepsilon$, since $|\delta'|\le \varepsilon$, we necessarily get $\delta\ge 2\varepsilon$.
Therefore 
\[|\delta+\delta'|\approx |2\delta+\delta'| \approx \delta \ge \delta ^2   \]
i.e.
\[|1-zw|\approx\left|1-z\frac{w+\bar w}{2}\right|\approx \varepsilon+\delta.\]
If $|2\delta+\delta'|\le 5\varepsilon$, then $|\delta|\le \frac {5\varepsilon}{2}$ and
\[|1-zw|\approx\left|1-z\frac{w+\bar w}{2}\right|\approx \varepsilon \approx \varepsilon+\delta.\]
\\
(b) Consider now the case where $w$ is far away from the real axis.
Let $\varphi:\hat \C \to \hat \C$ be the fractional linear transformation defined by 
\[z\mapsto \frac{1-z\frac{w+\bar w}{2}}{1-zw}.\]
The real axis is mapped by $\varphi$ to the circle $\mathcal C$ passing through $\varphi(0)=1$, $\varphi(1)=\frac{1-\frac{w+\bar w}{2}}{1-w} $, and $\varphi\left(\frac{2}{w+ \bar w}\right)=0$; the upper half-space is mapped inside the disc bounded by $\mathcal C$ (since, for instance, $\varphi\left( \frac{1}{\bar w} \right)=\frac 1 2 $). This directly implies that for any $z\in s(w)$, if $w=(1-\varepsilon)e^{i\delta}$ with $0\le\varepsilon\le \varepsilon_0$ and $\delta\ge \varepsilon_0$, we have $|\varphi(z)|\le c$ for some constant $c$.

\noindent Hence we are left to show that when $z\in s(w)$, $|\varphi(z)|$ is bounded away from $0$. 
Consider
\begin{align*}
\left|\varphi(z)-\frac 1 2 \right|=\left|\frac{1-z\frac{w+\bar w }{2}}{1-zw}-\frac 1 2 \right|=\frac{1}{2}\left|\frac{1-z\bar w }{1-zw}\right|.
\end{align*}
Let $\varepsilon_0$ be sufficiently small and $\frac{\pi}{2}\ge\delta \ge \delta_0$. For any $w=(1-\varepsilon)e^{i\delta}$ and for any $z\in s(w)$ we have that $\left|z -\frac{1}{\bar w}\right| \le 
\frac 1 2 \left|z -\frac{1}{ w}\right|$ which implies  $\left|\varphi(z)-\frac 1 2 \right|\le \frac 1 4$ and hence 
$|\varphi(z)|\ge 1/4$.
\\
(c) The last case to consider is when $|w|\le 1-\varepsilon _0$.
We have the following estimates
\begin{align*}
\frac{1-|w|}{1+|w|}\le \left|\frac{1-z\frac{w+\bar w}{2}}{1-zw}\right|\le \frac{1+|w|}{1-|w|},
\end{align*}
which, setting $r_0=1-\varepsilon_0$, lead to  
\[\frac{1-r_0}{1+r_0}\le |\varphi(z)|\le \frac{1+r_0}{1-r_0},\]
thus concluding the proof.
\end{proof}
The announced characterization result can then be stated as follows.
\begin{teo}
A measure $\mu$ on $\B$ is a Carleson measure for the Hardy space $H^2(\B)$ if and only if for any $q=re^{J\theta}\in \B$, the measure of the symmetric box $S(q)$ satisfies 
\[\mu(S(q))\lesssim |A_I(q)|,\] 
where $I$ is any imaginary unit and $|A_I(q)|$ denotes the length of the arc $A_I(q)$.
\end{teo}
\begin{proof}
Consider first the case in which $\supp \mu \subseteq (\B\cap \rr)$. In this case $\mu$ is a Carleson measure if and only if it is slice Carleson: since $\B \cap \rr$ is contained in each slice $\B_I$ we have
\[\int_{\B_I}|f(z)|^2d\mu(z)=\int_{\B}|f(q)|^2d\mu(q)\lesssim ||f||^2_{H^2(\B)}\]
for any $I \in \s$. Therefore the statement follows by Proposition \ref{charslice}.

Let now $\mu (\B\cap \rr)=0$.  
Suppose first that $\mu$ is Carleson for $H^2(\B)$ and let $w=u+v{I_w}\in \B$.
Consider the function 
\[K(q):=\frac{1}{A(\s)}\int_{\s}k_{u+vI}(q)dA_{\s}(I)\]
where $k_{u+vI_w}(q)=k_w(q)=(1-q\bar w)^{-*}$ is the reproducing kernel of $H^2(\B)$ and $dA_{\s}$ denotes the usual area element on the sphere $\s$.
Then, using the fact that $k_w(q)=\overline{k_q(w)}$ for any $w,q \in \B$ and the Representation Formula, we can write 
\begin{align}
 K(q)&=\frac{1}{A(\s)}\int_{\s}k_{u+vI}(q)dA_{\s}(I)=\frac{1}{A(\s)}\int_{\s}\overline{k_q(u+vI)} dA_{\s}(I) \nonumber \\
&=\frac{1}{A(\s)}\int_{\s}\left( \overline{\frac{1-IJ}{2}k_q(u+vJ)+\frac{1+IJ}{2}k_q(u-vJ)}\right)dA_{\s}(I)\nonumber \\
& =\frac{1}{A(\s)}\left(\int_{\s}\overline{k_q(u+vJ)} \frac{1-JI}{2}dA_{\s}(I)+\int_{\s}\overline{k_q(u-vJ)}\frac{1+JI}{2}dA_{\s}(I)\right) \nonumber \\
&=\frac{1}{2}\left(k_{u+vJ}(q)+k_{u-vJ}(q) \right) \label{kappa}
\end{align}
where $J$ is any imaginary unit.
The function $K(q)\in H^2(\B)$ since it is the superposition of functions in $H^2$ and, for any $J\in\s$,
\[\|K\|_{H^2(\B)}\le \frac{1}{2}\left(\|k_{u+vJ}\|_{H^2(\B)}+\|k_{u-vJ}\|_{H^2(\B)}\right)=\|k_{u+vJ}\|_{H^2(\B)}.\] 
Then, since $\mu$ is a Carleson measure, we can write
\begin{equation}\label{maggiorazione}
\int_{\B}|K(q)|^2d\mu(q) \lesssim \|K\|^2_{H^2(\B)}\le \|k_{u+vJ}\|^2_{H^2(\B)}=\frac{1}{1-(u^2+v^2)}=\frac{1}{1-|w|^2}\le \frac{1}{1-|w|}.
\end{equation} 
Let us study the modulus $|K(q)|$, when $q$ belongs to the symmetric box $S(w)$.
Recalling equation \eqref{kappa} together with the fact that we can choose the imaginary unit $J$ appearing in \eqref{kappa} to be the same as the one of $q$, we are left to estimate the modulus of $K$ slicewise, namely we need to estimate the quantity 
\[|K(z)|=\frac{1}{2}\left|k_w(z)+k_{\bar w}(z) \right|=\frac{1}{2}\left| \frac{1}{1-z\bar w}+\frac{1}{1-zw} \right|=\left| \frac{1-z\frac{w+\bar w}{2}}{1-zw}\right|\left|\frac{1}{1-z\bar w} \right|\]
with $z\in S_{I_w}(w)$.
Thanks to Lemma \ref{stimamoebius} and to classical estimates of the reproducing kernel of the complex Hardy space (see e.g. the proof of Theorem 9.3 in \cite{Duren}) we obtain that for any $J\in \s$
\[|K(x+yJ)|\gtrsim \frac{1}{1-|w|^2}.\]
With the notations of \eqref{krypton}, 
we get
\begin{align}\label{minorazione}
\int_{\B}|K(q)|^2d\mu(q)&=\int_{\s}\int_{\B^+_{I}}|K(x+yI)|^2d\mu^+_I(x+yI)d\nu(I)\nonumber  \\
&\ge \int_{\s}\int_{S_{I}(w)}|K(x+yI)|^2d\mu^+_I(x+yI)d\nu(I)  \nonumber \\
&\gtrsim \int_{\s} \int_{S_{I}(w)} \frac{1}{(1-|w|^2)^2}d\mu^+_I(x+yI)d\nu(I) = \frac{\mu(S(w))}{(1-|w|^2)^2}.  \tag{3.6}
\end{align}
Comparing equations (3.5) and \eqref{minorazione} we conclude
\[\mu(S(w))\lesssim \frac{(1-|w|^2)^2}{1-|w|}\le  
1-|w| \approx |A_{I}(w)|. \]

Let us now suppose that $\mu$ is a measure on $\B$ such that for any symmetric box $S(q)$, $\mu(S(q))\lesssim |A_{I}(q)|$. 
With the notation of (3.2), 
the hypothesis on $\mu$ yields that, for any $w\in \B$,
\begin{align}
|A_{I}(w)|&\gtrsim \mu(S(w))=\int_{S(w)}d\mu(q)=\int_{\s}\Big(\int_{S_{I}(w)}d\mu^+_I(z)\Big)d\nu(I)\nonumber \\ 
&=\int_{\s}\Big(\int_{S_{J_0}(u+vJ_0)}d{\mu^+}^{proj}_I(x+yJ_0)\Big)d\nu(I)\nonumber
\end{align}
where $J_0$ is any (fixed) imaginary unit, $S_{J_0}(w)$ is the projection of the symmetric box $S(w)$ onto the fixed semi-disc $\B^+_{J_0}$ and ${\mu^+}^{proj}_I$ is the projection of the measure $\mu^+_I$ on the same slice,
\[{\mu^+}_I^{proj}(E)=\mu^+_I(\{x+yI \, \colon \, y>0 \text{ and } x+yJ_0 \in E\}), \ \ d{\mu^+}^{proj}_I(x+yJ_0)=d\mu^+_I(x+yI)\] 
for any $E\subseteq \B^+_{J_0}$. 
Then the measure 
\[\int_{\s}d{\mu^+}^{proj}_I(x+yJ_0)d\nu(I)\] 
is Carleson for $H^2(\B_{J_0})$.
Let $f\in H^2(\B)$.
Using the Representation Formula, if $I$ denotes the imaginary unit of $q$, and $J_0$ is any imaginary unit, $J_0\neq \pm I$, we have
\begin{equation*}
\begin{aligned}
&\int_{\B}|f(q)|^2d\mu(q)
\lesssim \int_{\s}\int_{\B^+_{I}}\left(|f(x+yJ_0)|^2+|f(x-yJ_0)|^2\right)d\mu^+_I(x+yI)d\nu(I)\\
&= \int_{\B^+_{J_0}}|f(x+yJ_0)|^2\int_{\s}d{\mu^+}^{proj}_I(x+yJ_0)d\nu(I)+ \int_{\B^+_{J_0}}|f^c(x+yJ_0)|^2\int_{\s}d{\mu^+}^{proj}_I(x+yJ_0)d\nu(I),
\end{aligned}
\end{equation*}
where we used the fact that $|f(\bar q)|=|\overline{f(\bar q)}|=|f^c(q)|$.
Since the regular conjugate $f^c(q)=\sum_{n\ge 0}q^n\overline{a_n}$ of a function $f(q)=\sum_{n\ge 0}q^n a_n \in H^2(\B)$ belongs to $H^2(\B)$ as well (see [7]) and has same $H^2$-norm, thanks to Proposition 3.3 and to the fact that the $H^2$-norm is the same on each slice, we thus conclude
\[\int_{\B}|f(q)|^2d\mu(q)\lesssim \| f\| ^2_{H^2(\B_{J_0})}+\| f^c\| ^2_{H^2(\B_{J_0})}= 2\| f\| ^2_{H^2(\B)}.\]

\end{proof}

\section{Hankel bilinear forms}\label{hankelsection}
Let $b:\B \to \HH$ be a regular function. The {\em Hankel operator} associated with the symbol $b$ is the bilinear operator $T_b: H^2(\B)\times H^2(\B)\to \HH$ defined by
\[T_b(f,g)=\langle f\star g, b \rangle_{H^2(\B)}\]
for any $f,g \in H^2(\B)$. 

As stated in the Introduction, a natural problem in this setting is to find necessity and sufficiency conditions on the symbol $b$ so that the corresponding $T_b$ is a bounded operator,  i.e. so that there exists a constant $c(b)$ depending only on $b$ such that 
\[|T_b(f,g)|\le c(b) ||f||_{H^2(\B)}||g||_{H^2(\B)}\]
for any $f,g\in H^2(\B)$.

First, let us show that the boundedness of $T_b$ is equivalent to the boundedness of the operator $\Gamma_\alpha$  with $\alpha=\{ \overline{\check b}_n \}$, thus proving the first equivalence of Theorem \ref{eboli}. 
\begin{pro}
Let $b(q)=\sum_{n\ge 0}q^n\hat{b}_n$ be a regular function on $\B$ and let $\alpha=\{\alpha_n\}_{n}$ be a quaternionic sequence such that $\overline{\alpha_n}=\check b_n $ for any $n$. Then the operator $T_b: (f,g)\mapsto \langle f\star g, b \rangle_{H^2(\B)}$ is bounded on $H^2(\B) \times H^2(\B)$ if and only if the operator \[\Gamma_\alpha: \{a_{n}\}_n \mapsto \left\{\sum_{k\ge 0}\alpha_{n+k} a_k\right\}_n\] is bounded on $\ell^2(\N, \HH)$. 
Moreover 
\[
 \left\|\Gamma_\alpha\right\|_{\BBB(\ell^2(\N,\HH))}= \sup_{f,g\ne0}\frac{\left|\left<f\star g,b\right>_{H^2(\BB)}\right|}{\|f\|_{H^2(\BB)}\cdot\|g\|_{H^2(\BB)}}.\]
 \end{pro}
 \begin{proof}
Let $G_{\alpha}$ be the bilinear operator associated with $\Gamma_{\alpha}$, defined 
on $(c=\{c_n\}_{n\in \N}, d=\{d_n\}_{n\in \N})$ in $\ell^2(\N, \HH)\times\ell^2(\N, \HH)$ by    

 \[G_{\alpha}(c,d):=\left\langle d ,  \overline{\Gamma_{\alpha} c}\right\rangle_{\ell^2(\N, \HH)}=\sum_{n\ge 0}\sum_{k\ge 0 }\alpha_{n+k} c_k d_n .\]
 It is not difficult to see that $G_{\alpha}$ is bounded if and only if $\Gamma_{\alpha}$ is, and 
 \[\sup_{c,d\ne 0}\frac{\left|G_\alpha(c,d)\right|}{\|c\|_{\ell^2(\N, \HH)}\cdot\|d\|_{\ell^2(\N,\HH)}}=\|\Gamma_{\alpha}\|_{\mathcal B(\ell^2(\N, \HH))}.\]
 Let $f(q):=\check{c}(q)=\sum_{n\ge 0}q^nc_n$ and $g(q):=\check{d}(q)=\sum_{n\ge 0}q^nd_n$ be functions in $H^2(\B)$. Then
 \begin{equation}\label{minore}
 G_{\alpha}(c,d)=\sum_{n\ge 0 }\sum_{j\ge 0 }\alpha_{j} c_{j-n}d_n =\sum_{j\ge 0 }\alpha_{j}\sum_{n= 0 }^j c_{j-n} d_n =\left\langle  f\star g, b \right\rangle_{ H^2(\B)}
 \end{equation}
 and 
 \[\sup_{c,d \ne 0}\frac{\left|G_\alpha(c,d)\right|}{\|c\|_{\ell^2(\N, \HH)}\cdot\|d\|_{\ell^2(\N,\HH)}}=\sup_{f,g\ne0}\frac{\left|\left<f\star g,b\right>_{H^2(\BB)}\right|}{\|f\|_{H^2(\BB)}\cdot\|g\|_{H^2(\BB)}}\]
 which concludes the proof.
 \end{proof}

When a viable factorization theory is not available, and this is one such case, a more general tool was developed in the important article \cite{coifmanrochbergweiss}. Consider the weak $\star $-product space $H^2(\B)\prodstella H^2(\B)$, namely the space of all linear combinations of $\star $-products of pairs of functions in $H^2(\B)$, 
\[H^2(\B)\prodstella  H^2(\B)=\Big\{ \Phi:\B\to \HH \ \colon \ \Phi(q)=\sum_{j}f_j\star g_j(q), \text{ with } f_j, g_j\in H^2(\B)\Big\}\]
endowed with the norm
\[\| \Phi\prodstellanorma=\inf\Big\{\sum_j\| f_j\| _{H^2(\B)}\| g_j\| _{H^2(\B)} \ \colon \ \Phi=\sum_{j}f_j\star g_j, \text{ with } f_j,g_j\in H^2(\B)\Big\}.\]

The first step in studying the boundedness of the operator $T_b$ is the following.
\begin{teo}\label{bduale} 
Let $b:\B\to \HH$ be a regular function and set $\Lambda_b(h):=\left\langle h,b\right\rangle_{H^2(\B)}$ when $h$ is regular in $\B$. 
Then, there exists a constant $c(b)$, depending only on $b$, such that
\[|T_b(f,g)|\le c(b)\| f\| _{H^2(\B)}\| g\| _{H^2(\B)}\]
for any $f,g\in H^2(\B)$, if and only if $\Lambda_b$ belongs to the dual space $\left(H^2(\B)\prodstella H^2(\B)\right)^{*}$.

On the other hand, for all $\Lambda$ in $\left(H^2(\B)\prodstella H^2(\B)\right)^{*}$ there exists a unique regular $b:\B\to\HH$ such that 
$\Lambda=\Lambda_b$.
Moreover,
$$
\|\Lambda_b\prodstellanormaduale\approx\sup_{f\ne0\ne g\in H^2(\B)}\frac{|T_b(f,g)|}{\|f\|_{H^2(\B)}\cdot \|g\|_{H^2(\B)}}
$$
\end{teo}
\begin{proof}
The argument is standard. We translate it into quaternionic language for the ease of the reader. First consider a regular $b$ such that $\Lambda_b\in \left(H^2(\B)\prodstella H^2(\B)\right)^{*}$. Then, for any $f,g\in H^2(\B)$,
\begin{equation*}
\left|\langle f\star g, b\rangle_{H^2(\B)}\right|\le\| f\star g\prodstellanorma\| \Lambda_b\prodstellanormaduale\le\| f\| _{H^2(\B)}\| g\| _{H^2(\B)}\| \Lambda_b\prodstellanormaduale
\end{equation*}
where the last equality follows from the definition of the norm on $H^2(\B)\prodstella H^2(\B)$. 
Let now $\Lambda$ be a linear functional in $\left(H^2(\B)\prodstella H^2(\B)\right)^{*}$. Since $H^2(\B)$ imbeds in $H^2(\B)\prodstella H^2(\B)$ continuously, as a dense subspace, 
$\Lambda$ can be identified with a unique element of $(H^2(\B))^*=H^2(\B)$. Hence, $\Lambda=\Lambda_b$ for some $b$ in  $H^2(\B)$, and
$$
\sup_{f\ne0\ne g\in H^2(\B)}\frac{|T_b(f,g)|}{\|f\|_{H^2(\B)}\cdot \|g\|_{H^2(\B)}}\lesssim\|\Lambda_b\prodstellanormaduale.
$$ 
In the other direction, let $\Phi\in H^2(\B)\prodstella H^2(\B)$. Then, for any decomposition of $\Phi$ of the form $\Phi=\sum_{j}f_j\star g_j$, if $T_b$ is bounded, we have
\begin{equation*}
|\langle \Phi,b \rangle_{H^2(\B)}|\le \sum_{j}\left|\langle f_j\star g_j,b \rangle_{H^2(\B)}\right|\le \sum_{j}\| f_j\| _{H^2(\B)}\| g_j\| _{H^2(\B)} c(b)
\end{equation*}
for some constant $c(b)$ depending only on the function $b$.
Therefore, taking the infimum on all possible decompositions of the function $\Phi$, we get
\[|\langle \Phi,b \rangle_{H^2(\B)}|\le \inf \Big\{ \sum_{j}\| f_j\| _{H^2(\B)}\| g_j\| _{H^2(\B)} c(b) \ \ |\ \ f_j,g_j\in H^2(\B) \Big\}=c(b)\| \Phi\prodstellanorma \]
thus showing that $\Lambda_b\in \left(H^2(\B)\prodstella H^2(\B)\right)^{*}$ and that 
$$
\sup_{f\ne0\ne g\in H^2(\B)}\frac{|T_b(f,g)|}{\|f\|_{H^2(\B)}\cdot \|g\|_{H^2(\B)}}\gtrsim\|\Lambda_b\prodstellanormaduale.
$$ 
\end{proof}

Another way to characterize the functions $b$ for which the Hankel operator $T_b$ is a bounded operator, is in terms of Carleson measures. 

First, we need a preliminary result (which holds in analogy with the complex case), concerning the $H^2$-norm. The calculations proving it are identical to those of the complex case.
\begin{pro}\label{norma2}
Let $dVol_{\B}$ denote the volume form on $\B$ defined by  $dVol_{\B}(x+yI)=\frac{1}{4}dA_{\s}(I)dxdy$. Then, for any $f\in H^2(\B)$
\[\| f\| ^2_{H^2(\B)}=|f(0)|^2+\frac{1}{Vol(\B)}\int_{\B}|\de f(q)|^2 \log |q|^{-2}dVol_{\B}(q)\approx |f(0)|^2+\frac{1}{Vol(\B)}\int_{\B}|\de f(q)|^2(1-|q|^2)dVol_{\B}(q).\]
\end{pro}
For a proof of the equivalence between the complex analogues of the two integral terms $\frac{1}{Vol(\B)}\int_{\B}|\de f(q)|^2 \log |q|^{-2}dVol_{\B}(q)$ and $\frac{1}{Vol(\B)}\int_{\B}|\de f(q)|^2(1-|q|^2)dVol_{\B}(q)$, see, e.g., the proof of Theorem $8.1.10$ in \cite{Zhu}.
By polarization, we can use as an inner product of $H^2(\B)$ the following
\begin{equation}\label{innerpro2}
\langle f, g\rangle_{H^2(\B)} = \overline{g(0)}f(0)+\frac{1}{Vol(\B)}\int_{\B}\overline{\de g(q)}\de f(q)(1-|q|^2)dVol_{\B}(q).\end{equation}

The following result can be interpreted as an analogue of the Nehari Theorem in the quaternionic setting.
\begin{teo}\label{bcarleson}
Let $b:\B\to \HH$ be a regular function. Then the Hankel operator $T_b$ is bounded, i.e. there exists a constant $c(b)$ depending only on $b$ such that 
\begin{equation}\label{hankel}
\left|\langle f\star g, b\rangle_{H^2(\B)}\right|^2 \le c(b)\| f\| ^2_{H^2(\B)}\| g\| ^2_{H^2(\B)}
\end{equation}
for any $f,g\in H^2(\B)$, if and only if the measure on the unit ball $\B$ defined by
\begin{equation}\label{measure}
|\de b(q)|^2(1-|q|^2)dVol_{\B}(q)
\end{equation}
is a Carleson measure for $H^2(\B)$. 
\end{teo}
\begin{proof}
Suppose first that $T_b$ is bounded. Let $I\in \s$ and consider any $f,g\in H^2(\B)$ such that $f,g:\B_I\to L_I$, namely such that $f,g\in H^2(\D)\subset H^2(\B_I)$; according to the Splitting Lemma, let $b$ be decomposed as
\[b(z)=b_1(z)+b_2(z)J\]
with $b_1,b_2:\B_I\to L_I$ holomorphic and $J\in \s$ orthogonal to $I$.
Then, recalling that the norm and the inner product of $H^2(\B)$ can be computed on any slice,
\begin{equation*}
|\langle f\star g, b\rangle_{H^2(\B)}|^2=|\langle f\star g, b\rangle_{H^2(\B_I)}|^2=|\langle fg, b_1+b_2J\rangle_{H^2(\B_I)}|^2=|\langle fg, b_1\rangle_{H^2(\D)}|^2+|\langle fg,b_2\rangle_{H^2(\D)}|^2
\end{equation*}
which, recalling equation \eqref{hankel}, leads to
\[|\langle fg,b_k\rangle_{H^2(\D)}|^2\le c(b)\| f\| ^2_{H^2(\B)}\| g\| ^2_{H^2(\B)}\]
for both $k=1,2$. Thanks to the analogous result (see e.g. \cite{Zhu}) in the complex case, we get that $|b_k'(z)|^2(1-|z|^2)dxdy$ is a (complex) Carleson measure for $H^2(\D)\subset H^2(\B_I)$ for $k=1,2$ and hence, for any $f\in H^2(\B_I)$, 
\begin{align*}
&\int_{\B_I}|f(z)|^2|\de b(z)|^2(1-|z|^2)dxdy=\int_{\B_I}|f(z)|^2(|b_1'(z)+b_2'(z)J|^2)(1-|z|^2)dxdy\\
&=\int_{\B_I}|f(z)|^2|b_1'(z)|^2(1-|z|^2)dxdy+\int_{\B_I}|f(z)|^2|b_2'(z)|^2(1-|z|^2)dxdy \\
&\lesssim ||f||^2_{H^2(\B_I)}
\end{align*}
which, recalling (\ref{slicecarleson}), implies that the quaternionic measure 
$|\de b(z)|^2(1-|z|^2)dxdy$
is Carleson for  $H^2(\B_I)$. 
The fact that $I$ was any imaginary unit, yields that the measure $|\de b(q)|^2(1-|q|^2)dVol_{\B}(q)$ is slice Carleson for $H^2(\B)$ and therefore, thanks to Proposition \ref{sliceimplica}, Carleson.

Let us now prove the opposite implication. 
Using Equation \eqref{innerpro2} and the Leibniz rule for the slice derivative, if $z=x+yI\in \B_{I}$, we can write \begin{equation*}
\begin{aligned}
&|\langle f\star g, b\rangle_{H^2(\B)}| \lesssim \left|\overline{b(0)}(f\star g)(0)\right|+\left|\int_{\B}\overline{\de b (q)} \de (f\star g)(q)(1-|q|^2)dVol_{\B}(q)\right|\\
&\le \left|\frac 1 4 \int_{\s}dA_{\s}(I)\int_{\B_{I}}\overline{\de b (z)} \de (f\star g)(z)(1-|z|^2)dxdy\right|\\ 
&\lesssim \int_{\s}dA_{\s}(I) \int_{\B_{I}} |\de b(z)\| \de f(z)\| g(\hat z)|(1-|z|^2)dxdy \\ &\hskip 2 cm + \int_{\s}dA_{\s}(I)\int_{\B_{I}} |\de b(z)\| f(z)\| \de g(\tilde z)|(1-|z|^2)dxdy = \mathcal I_1+ \mathcal I_2
\end{aligned}
\end{equation*}
where $\hat z$ and $\tilde z$ are points lying on the same two sphere as $z$, determined in view of the expression  of the $\star $-product in terms of the pointwise product (see Proposition \ref{trasf}. Here we are omitting the discrete subset of $\B_{I}$ where $f$ or $\de f$ vanish).
To estimate the integral $\mathcal I_1$ we first use the Cauchy-Schwarz inequality, and then Proposition \ref{norma2} to get
\begin{equation}\label{Iquadro}
\begin{aligned}
\mathcal I_1^2 &\lesssim \int_{\s}dA_{\s}(I)\int_{\B_{I}} |\de b(z)|^2|g(\hat z)|^2(1-|z|^2)dxdy\int_{\s}dA_{\s}(I)\int_{\B_{I}}|\de f(z)|^2(1-|z|^2)dxdy\\
&\lesssim \| f\| ^2_{H^2(\B)}\int_{\s}dA_{\s}(I) \int_{\B_{I}} |\de b(z)|^2|g(\hat z)|^2(1-|z|^2)dxdy.
\end{aligned}
\end{equation}
%
The Representation Formula allows us to express $g(\hat z)$ in terms of $g(z)$ and $g(\bar z)$, so that we get
\begin{equation}\label{Iquadro2}
\begin{aligned}
\int_{\s}dA_{\s}(I)\int_{\B_{I}} &|\de b(z)|^2|g(\hat z)|^2(1-|z|^2)dxdy\\
&\lesssim \int_{\s}dA_{\s}(I)\int_{\B_{I}} |\de b(z)|^2(|g(z)|^2+|g(\bar z)|^2)(1-|z|^2)dxdy.
\end{aligned}
\end{equation}
Now $|g(\bar z)|=|\overline{g(\bar z)}|=|g^c(z)|$ and $\| g^c\| _{H^2(\B)}=\| g\| _{H^2(\B)}$ (see, e.g., \cite{hardy}); hence, using equations \eqref{Iquadro}, \eqref{Iquadro2} and recalling that $|\de b(q)|^2(1-|q|^2)dVol_{\B}(q)$ is a Carleson measure for $H^2(\B)$, we conclude
\[\mathcal I_1^2 \lesssim \| g\| ^2_{H^2(\B)}\| f\| ^2_{H^2(\B)}.\] 
With analogous arguments, we get the same estimate on the second integral $\mathcal I_2$, thus completing the proof. 
\end{proof}

\section{Quaternionic BMOA}\label{gorizia}
\begin{defn}
Let $f\in H^1(\B)$ and, for any interval $a=(\alpha, \beta)$ of $\rr$ such that $|a|:=|\beta-\alpha|\le 2\pi$, denote by $f_{I,a}$ the average value of (the radial limit of) $f$ on the arc $(e^{\alpha I}, e^{\beta I})\subseteq \p \B_I$,
\[f_{I,a}=\frac{1}{|a|}\int_{a}f(e^{\theta I})d\theta.\]
We say that $f\in BMOA(\B_I)$ if 
\[\| f\| _{BMOA(\B_I)}:=\sup_{{a\subset \rr, \ |a| \le 2\pi}}\left\{\frac{1}{|a|}\int_{a}|f(e^{\theta I})-f_{I,a}|d\theta\right\}<+\infty.\] 
We say that $f \in BMOA (\B)$ if 
\[\| f\| _{BMOA(\B)}:=\sup_{I\in \s}\| f\| _{BMOA(\B_I)}<+\infty.\] 
\end{defn}
\noindent Since $\| c \| _{BMOA(\B)}=0$ for any constant function $c$, the quantity $\| \cdot \| _{BMOA(\B)}$ defines only a semi-norm on $BMOA(\B)$. To make it a norm it suffices to consider 
$\| f \| _{BMOA(\B)}+ f(0)$ for any $f\in{BMOA(\B)}$. For the purposes of the present paper we can work with the semi-norm.\\
The fact that $f\in H^1(\B)$ yields that almost everywhere at the boundary $f$ satisfies the Representation Formula (see \cite{hardy}), which leads to the following statement.
\begin{pro}\label{BMOslice}
Let $f\in H^1(\B)$. Then $f\in BMOA(\B)$ if and only if $f\in BMOA(\B_I)$ for some $I\in \s$. More precisely
\[\| f\| _{BMOA(\B_I)}\le\|f\|_{BMOA(\B)}\le 2\|f\|_{BMOA(\B_I)}.\]

\end{pro} 
\begin{proof}
The necessity condition is trivial. 
Suppose then that $f\in BMOA(\B_I)$ for $I\in \s$. 
Thanks to the Representation Formula, the average value of $f$ on a different slice $\B_J$, $J\in\s$, $J\neq\pm I$, can be computed as
%
\begin{equation*}
\begin{aligned}
f_{J,a}&=\frac{1}{|a|}\int_{a}f(e^{\theta J})d\theta=\frac{1}{|a|}\int_{a}\left(\frac{1-JI}{2}f(e^{\theta I})+\frac{1+JI}{2}f(e^{-\theta I})\right)d\theta=\frac{1-JI}{2}f_{I,a}+\frac{1+JI}{2}f_{I,-a}
\end{aligned}
\end{equation*}
where, if $a=(\alpha,\beta)$, then $-a=(-\beta,-\alpha)$.
The statement follows by simple calculations. 
\end{proof}

\begin{teo}\label{BMOACarleson}
Let $f\in H^1(\B)$. Then $f\in BMOA(\B)$ if and only if the measure $|\de f(q)|^2(1-|q|^2)dVol_{\B}(q)$ is a Carleson measure for $H^2(\B)$.
\end{teo}
\begin{proof}
Let $I,J\in\s$ with $J \perp I$, and consider the splitting of $f$ on $\B_I$ with respect to $J$, $f=F+GJ$, where $F,G:\B_I\to L_I$ are holomorphic functions.
Notice that for any $a \subset \rr, |a|\le 2\pi$,
\[f_{I,a}=\frac{1}{|a|}\int_{a}f(e^{\theta I})d\theta=\frac{1}{|a|}\int_{a}F(e^{\theta I})+G(e^{\theta I})Jd\theta=F_{a}+G_aJ\]
where $F_a,G_a$ are the average values of $F$ and $G$ respectively.
Moreover, thanks to the orthogonality of $I$ and $J$ we get
\begin{equation}\label{BMOAsplit}
\begin{aligned}
\frac{1}{|a|}\int_{a}|f(e^{\theta I})-f_{I,a}|d\theta=\frac{1}{|a|}\int_{a}\left(|F(e^{\theta I})-F_{a}|^2+|G(e^{\theta I})-G_a|^2\right)^{1/2}d\theta.
\end{aligned}
\end{equation}
Then, if $f\in BMOA(\B)$, we also have 
that both $F$ and $G$ belong to the (complex) space $BMOA(\D)$.
Thanks to classical results (see, e.g., Theorem 8.3.5 in \cite{Zhu}) we get then that both $|F'(z)|^2(1-|z|^2)dxdy$ and $|G'(z)|^2(1-|z|^2)dxdy$
are Carleson measures for the complex Hardy space $H^2(\D)\subset H^2(\B_I)$. Hence
\[|\de f(z)|^2(1-|z|^2)dxdy=\left(|F'(z)|^2+|G'(z)|^2\right)(1-|z|^2)dxdy\]
is a Carleson measure for $H^2(\B_I)$. Since $I$ was any imaginary unit, we have that $|\de f(q)|^2(1-|q|^2)dVol_{\B}(q)$ is slice Carleson for $H^2(\B)$ and therefore it is a Carleson measure for $H^2(\B)$.

If, on the other hand, $|\de f(q)|^2(1-|q|^2)dVol_{\B}(q)$ is a Carleson measure for $H^2(\B)$, thanks to 
Proposition \ref{bcarleson} we get that the Hankel operator $T_f$ associated with $f$
\[T_f:H^2(\B)\times H^2(\B) \to\HH, \quad (g_1,g_2)\mapsto \langle g_1\star g_2, f \rangle_{H^2(\B)} \]
is bounded. In particular, $T_f$ is bounded also when restricted to $H^2(\D)\times H^2(\D) \subset H^2(\B_I)\times H^2(\B_I)$. Namely, for any $g_1,g_2\in H^2(\B)$ that map $\B_I$ to $L_I$ for some $I\in \s$, if the splitting of $f$ on $\B_I$ with respect to $J \perp I$ is $f=F+GJ$, then
\[+\infty>|\langle g_1\star g_2, f\rangle_{H^2(\B)}|^2=|\langle g_1g_2, F+GJ\rangle_{H^2(\B_I)}|^2=|\langle g_1g_2, F\rangle_{H^2(\D)}|^2+|\langle g_1g_2, G\rangle_{H^2(\D)}|^2.\] 
Hence, thanks to classical results (see e.g. \cite{peller}) we have that both $F$ and $G$ belong to $BMOA(\D)$. Recalling equation \eqref{BMOAsplit} we get that $f=F+GJ \in BMOA(\B_I)$ which, by Proposition \ref{BMOslice}, leads to the conclusion.   
\end{proof}
Combining Theorem \ref{bduale}, Theorem \ref{bcarleson} and Theorem \ref{BMOACarleson} we get the following identification.
\begin{coro}\label{duale}
The two spaces $(H^2(\B)\prodstella H^2(\B))^*$ and $BMOA(\B)$ coincide.
\end{coro}

A natural question is then whether the spaces $H^1(\B)$ and  $H^2(\B)\prodstella H^2(\B)$ coincide or not. A partial answer is given by the following result.
\begin{pro}\label{due}
$H^1(\B)=H^2(\B)\star H^2(\B)+H^2(\B)\star H^2(\B)$.
\end{pro}
\begin{proof}
On the one hand $H^2(\B)\star H^2(\B)+H^2(\B)\star H^2(\B)\subseteq H^1(\B)$ since $H^2(\B)\star H^2(\B)\subseteq H^1(\B)$ as proven in \cite{hardy}.
On the other hand, let $f\in H^1(\B)$, $I,J\in \s$ with $I\perp J$ and let $F,G:\B_I\to L_I$ be such that $f_I=F+GJ$ on $\B_I$. Then, since $F,G\in H^1(\D)\subset H^1(\B_I)$ (see \cite{hardy}), thanks to the classical factorization for the complex Hardy space $H^1(\D)$ we can write $F(z)=F_1(z)F_2(z)$ and $G(z)=G_1(z)G_2(z)$ with $F_1,F_2,G_1,G_2 \in H^2(\D)\subset H^2(\B_I)$.
Hence, on the slice $\B_I$ we have the following expression of $f_I$
\[f_I(z)=F_1(z)F_2(z) + G_1(z)G_2(z)J=F_1\star F_2(z) + G_1\star G_2(z)J ,\]
which can be (uniquely) extended on the entire $\B$ by means of the Extension Lemma as
\[f(q)=\ext(F_1\star F_2 + G_1\star G_2\, J)(q)=\ext(F_1)\star \ext(F_2)(q) + \ext(G_1)\star \ext(G_2)(q)J \]
where the last equality follows from the Identity Principle.   
Since $F_1,F_2,G_1,G_2 \in H^2(\B_I)$ we get that their regular extensions belong to $H^2(\B)$ and therefore that $f\in H^2(\B)\star H^2(\B)+H^2(\B)\star H^2(\B)$.
\end{proof}
As a consequence
we have that $H^1(\B)\subseteq H^2(\B)\prodstella H^2(\B)$. 
\begin{teo}
The dual space of $H^1(\B)$ is $BMOA(\B)$.
\end{teo}

\begin{proof}
Corollary \ref{duale}, Proposition \ref{due}, and a standard density argument, yield that $BMOA(\B)\subseteq (H^1(\B))^*$.

To prove the opposite inclusion, let $L\in (H^1(\B))^*$ be a bounded linear functional on $H^1(\B)$. 
Then, for any $I\in \s$, $L\in (H^1(\B_I))^*$. Since, for any $I\in \s$, \, $H^1(\B_I)$ is a linear subspace of $L^1(\partial \B_I)$ (see \cite{hardy}), the Hahn-Banach Theorem 
in the quaternionic setting (see \cite{Brackx}) yields that $L$ extends to a bounded linear functional on $L^1(\partial \B_I)$. 
As in the complex case, $(L^1(\partial \B_I))^*=L^{\infty}(\partial \B_I)$, hence, there exists $\varphi\in L^{\infty}(\partial \B_I)$ such that
\[L(f)=\frac{1}{2\pi}\int_{-\pi}^{\pi}\overline{\varphi(e^{\theta I})}f(e^{\theta I})d\theta\]
for any $f\in L^1(\partial \B_I)$.
The same duality relation holds, restricted to $f\in H^1(\B)$, even if we replace $\varphi$ by its projection $g=P\varphi$ onto the functions which are holomorphic on $\B_I$, with 
values in the quaternions (here the Splitting Lemma is used); and it follows from one complex dimensional Nehari Theory that $g$ lies in $BMOA(\B_I)$.
Thanks to Proposition \ref{BMOslice} we conclude that $g\in BMOA(\B)$.  
\end{proof}
\begin{coro}\label{incrociamoledita}
 $H^1(\BB)=H^2(\BB)\prodstella H^2(\BB)$.
\end{coro}
\begin{proof}
 Equality of duals implies norm equivalence in the spaces.
\end{proof}

\section{Concluding remarks}\label{trento}
We conclude with some problems connected with the paper's subject which we could not answer.
\begin{itemize}
\item We know that $H^1(\BB)=H^2(\BB)\prodstella H^2(\BB)$.
It would be interesting to know if $H^2(\B)\star H^2(\B)=H^1(\B)$; that is if
 a good factorization theory exists in the quaternionic setting. 
\item Our testing function for the Carleson measure theorem is an average of reproducing kernels. It would be interesting to know if the ``reproducing kernel thesis'' holds;
that is if the inequality $\int_\B\left|k_w\right|^2d\mu\le c(\mu)\left\|k_w \right\|_{H^2(\B)}^2$, with $c(\mu)$ independent of $w$, implies that $\mu$ is Carleson for $H^2(\B)$.
\end{itemize}
\vskip 0.5 cm

\noindent{\small {\bf Acknowledgments.} 
We thank Zhenghua Xu (University of Hefei, China) for pointing out that the finiteness of the measure $\mu$ is needed in Proposition 3.1. Both authors are partially supported by the PRIN project ``Real and Complex Manifolds'' of the Italian MIUR. The first author is also partially supported by GNAMPA of the INdAM. The second author is also partially supported by GNSAGA of the INdAM, by  the FIRB project ``Differential Geometry and Geometric Function Theory'' and by the SIR project ``Analytic aspects in complex and hypercomplex geometry'' of the Italian MIUR.}


\end{document}